\documentclass[10pt]{amsart}
\usepackage[utf8]{inputenc}
\usepackage{graphicx}
\usepackage{amsfonts,amssymb}
\usepackage{hyperref}
\newtheorem{thm}{Theorem}[section]

\newtheorem{lem}[thm]{Lemma}

\theoremstyle{definition}

\theoremstyle{definition}
\newtheorem{rem}[thm]{Remark}

\numberwithin{equation}{section}

\renewcommand\Re{\operatorname{Re}}
\renewcommand\Im{\operatorname{Im}}
\usepackage{scalerel}
\usepackage{stackengine}
\stackMath
\def\hatgap{2pt}
\def\subdown{-2pt}
\newcommand\reallywidehat[2][]{%
\renewcommand\stackalignment{l}%
\stackon[\hatgap]{#2}{%
\stretchto{%
    \scalerel*[\widthof{$#2$}]{\kern-.6pt\bigwedge\kern-.6pt}%
    {\rule[-\textheight/2]{1ex}{\textheight}}
}{0.5ex}
_{\smash{\belowbaseline[\subdown]{\scriptstyle#1}}}%
}}
\begin{document}
\title[Fokas method for IBVPs involving mixed derivatives]{Fokas method for linear boundary value problems involving mixed spatial
derivatives}%
\author[A. Batal]{A.Batal}
\address[A. Batal]{Department of Mathematics, Izmir Institute of Technology, TR}%
\email{abatal@iyte.edu.tr}
\author[A.S. Fokas]{A.S. Fokas}
\address[A.S. Fokas]{DAMTP, University of Cambridge, UK}%
\email{tf227@cam.ac.uk}%
\author[T. {\" O}zsarı]{T. {\" O}zsarı*}%
\address[T. {\" O}zsarı]{Department of Mathematics, Izmir Institute of Technology, TR}%
\email{turkerozsari@iyte.edu.tr}%


\thanks{T. {\" O}zsar{\i} and A.Batal's research are supported by T\"{U}B\.{I}TAK 1001 Grant \#117F449 }%
\thanks{*Corresponding author (Türker Özsarı, turkerozsari@iyte.edu.tr)}%
\begin{abstract}
  We obtain solution representation formulas for some linear initial boundary value problems posed on the half space that involve mixed spatial derivative terms via the unified transform method (UTM), also known as the Fokas method.  We first implement the method on the second order parabolic PDEs; in this case one can alternatively eliminate the mixed derivatives by a linear change of variables.  Then, we employ the method to biharmonic problems, where it is not possible to eliminate the cross term via a linear change of variables. A basic ingredient of the UTM is the use of certain invariant maps.  It is shown here that these maps are well-defined provided that certain analyticity issues are appropriately addressed.
\end{abstract}
\maketitle
\tableofcontents
\section{Introduction}
The rigorous wellposedness analysis of solutions of linear and nonlinear partial differential equations (PDEs) in function spaces requires a suitable notion of \emph{solution}.  This is especially important when a PDE is studied at low regularity level where classical derivatives do not exist.   In other words, at low regularity level, one cannot verify whether a given function satisfies the given PDE in the pointwise sense. Nevertheless, there exist other mathematically precise notions of solution which do not require differentiability in the classical sense.  These are referred to as \emph{weak} solutions.  There are several well known ways to define such solutions.  The most classical one is to consider a PDE in the \emph{distributional} sense rather than pointwise sense, so that all the derivatives can be transferred onto smooth test functions. In this way, the sought after solution, which is a function or a generalized function, does not need to have derivatives itself in order to satisfy the given PDE.  Another approach that can be used for defining weak solutions to initial boundary value problems (IBVPs) is to employ \emph{representation formulas}.   These formulas can either be in an abstract functional analytic form, such as semigroup formulas, or can be explicit, such as those obtained via transform methods.  The explicit formulas are first obtained for smooth solutions under the assumption that the sought after solution and the given data are smooth and have sufficient decay properties. After an integral formula is obtained, it can then be used as the \emph{definition} of a generalized (weak) solution (indeed, the integral still makes sense under much weaker assumptions regarding the smoothness and decay assumptions of the data).  The unified transform method  (\cite{Decon14,Fok97,Fok02,Fok08}) provides one of the most efficient and general ways to obtain such integral formulas for defining solutions of IBVPs.  Most importantly, the integral representations obtained via UTM have a very suitable space-time structure allowing rigourous wellposedness analysis for linear and nonlinear initial boundary value problems; such results are obtained, for example, in \cite{Him17}, \cite{Tian17}, \cite{alex2018wellposedness}, and \cite{Ozs19} for Schrödinger and biharmonic Schrödinger type equations, in \cite{Him19} for Korteweg-de Vries equation, and in \cite{him15} for the ``good'' {B}oussinesq equation.

The present paper implements the UTM on PDEs with mixed spatial derivatives in 2D, and also shows that the analyticity issues arising from the appearance of square roots can be resolved. Until now, the implementation of the UTM on equations involving mixed spatial derivatives was limited to only a few papers, see \cite{fok09} and \cite{mant11}. Furthermore, the general application of the UTM even in 1D often involves the appearance of square roots.  The related analyticity issues must be handled with suitably chosen branch cuts. For instance in \cite{Decon14} and \cite{him15-2}, the authors chose branch cuts intersecting with the standard contour of integration of the UTM which motivated them to deform this contour. In comparison with these papers, we show that it is possible to choose a branch cut which is strictly away from the standard contour except at branch points. This allows us to justify a solution formula given on the standard contour with no deformation.  In particular, details are presented for three canonical problems on the half space: a second order parabolic equation, a fourth order heat equation, and the biharmonic Schrödinger equation, namely,
\begin{eqnarray}
  u_t-u_{xx}-2u_{xy}-u_{yy}&=0,\label{11}\\
  u_t+\Delta^2 u&=0\label{12} \\
  iu_t+\Delta^2 u&=0. \label{13}
\end{eqnarray}
If mixed derivatives are present, one may try to eliminate the cross terms via a linear change of variables. One drawback of this approach is that for higher order PDEs such as the fourth order heat equation and the biharmonic Schrödinger equation, there does not exist a linear change of variables that allows the elimination of the cross term(s). Another drawback is that, depending on the coefficients, the relevant change of variable may distort the simple geometry of the spatial domain, making it more difficult to work in the new geometry. In this paper, we treat PDEs with mixed spatial derivatives directly and present the details of the UTM for obtaining integral representation for the three canonical problems \eqref{11}-\eqref{13}.

The implementation of the UTM involves three main steps: (i) finding a global relation, which is an identity involving the Fourier transform of the solution and of the initial data, as well as suitable $t$-transforms of the unknown and known boundary values; (ii) finding maps which keep the spectral inputs of the $t$-transforms of the boundary traces invariant; and (iii) performing a subtle contour deformation.  The second and the last steps require analyticity of various integrands as well as analyticity of the associated invariance maps in some suitable regions in the complex plane. It turns out that for some evolution equations, such as the biharmonic Schrödinger equation, one encounters certain analyticity issues. For example, the invariance map defined via the standard square root function is not analytic on open domains containing the standard contour of integration constructed via the Fokas method. Furthermore, it is proven here that for the above equations it is not possible to find an invariance map which is analytic on the boundary of this contour.  We treat this issue by choosing a slightly rotated branch cut for the square root function, proving that this moves the domain of nonanaliticity of the invariance maps away from the contour of integration.  This issue is not present in the case of the fourth order heat equation because the domain of nonanaliticity of the invariance maps is already away from the contour of integration.

\section{Second order parabolic PDEs}
In this section, we implement the UTM on the second order parabolic equation on the half space, which involves a mixed spatial derivative term:
 \begin{align}
     \label{maineq1}u_t-u_{xx}-2u_{xy}-u_{yy}&=0,\quad (x,y,t)\in\mathbb{R}\times \mathbb{R}_+\times (0,T),\\
     \label{bc1}u(x,0,t)&=g_0(x,t), \quad (x,t)\in\mathbb{R}\times (0,T),\\
     \label{init1}u(x,y,0)&=u_0(x,y), \quad \quad (x,y)\in\mathbb{R}\times\mathbb{R}_+.
     \end{align}
We begin by defining the \emph{half space Fourier transform} in the variables $(x,y)$:
\begin{equation}\label{FT}
  \hat{u}(\kappa,\lambda,t) := \int_{-\infty}^\infty\int_{0}^\infty \exp(-ix\kappa-iy\lambda)u(x,y,t)dydx, \quad \kappa\in \mathbb{R}, \Im\lambda\le 0.
\end{equation}
In what follows, we will assume that $u$ is sufficiently smooth in $\Omega_T=\mathbb{R}\times \mathbb{R}_+\times (0,T)$ up to the boundary of $\Omega_T$, and also that u decays sufficiently fast as $|(x,y)|\rightarrow \infty$.  We treat the spectral variables $\kappa$ and $\lambda$ as real and complex, respectively.  Note that the condition $\Im\lambda\le 0$ is essential for \eqref{FT} to make sense.  Applying the half space Fourier transform to
\eqref{maineq1}-\eqref{init1} and integrating by parts, we obtain
\begin{equation}\label{FTofMaineq}
\begin{split}
\hat{u}_t(\kappa,\lambda,t) = -(\kappa+\lambda)^2\hat{u}(\kappa,\lambda,t)-\reallywidehat[x]{g_1}(\kappa,t)-i(\lambda+2\kappa) \reallywidehat[x]{{g}_0}(\kappa,t),
\end{split}
\end{equation} where $g_1(x,t):=u_y(x,0,t)$ (an unknown trace) and $\reallywidehat[x]{(\cdot)}$ denotes the classical Fourier transform in the $x$ variable, i.e.,
\begin{equation}\label{gjhat}
  \reallywidehat[x]{g_j}(\kappa,t):= \int_{-\infty}^\infty\exp(-ix\kappa)g_j(x,t)dx.
\end{equation}
The solution of \eqref{FTofMaineq} together with the half space Fourier transform of \eqref{init1}, namely,
\begin{equation}\label{FTofInit}
  \hat{u}(\kappa,\lambda,0) = \widehat{u_0}(\kappa,\lambda),
\end{equation} yield the following equation:
\begin{equation}\label{sol1}
\begin{split}
  &\hat{u}(\kappa,\lambda,t)\exp((\kappa+\lambda)^2t)\\
  &= \widehat{u_0}(\kappa,\lambda)-\widetilde{g_1}(\kappa,(\kappa+\lambda)^2,t)-i(\lambda+2\kappa) \widetilde{g_0}(\kappa,(\kappa+\lambda)^2,t),\quad \kappa\in \mathbb{R}, \Im\lambda\le 0,
  \end{split}
\end{equation}where $\widetilde{g_j},j=0,1$, denote the \emph{t-transforms} of the known and unknown boundary values:
 \begin{equation}\label{ttransform}
    \widetilde{g_j}(\kappa,\eta,t)=\int_0^t\exp(\eta \tau)\reallywidehat[x]{g_j}(\kappa,\tau)d\tau, \quad j=0,1.
 \end{equation}
 The identity \eqref{sol1}, which relates the half space Fourier transform of the solution, the half space Fourier transform of the initial datum, and the $t$-transforms of known and unknown boundary traces, is referred to as a \emph{global relation}.
\begin{rem}
 An alternative way to obtain equation \eqref{sol1} directly is to use the divergence theorem.  Indeed, let $v$ be a solution of the formal adjoint equation of equation \eqref{maineq1}, namely let $v$ be a solution of the equation
 \begin{equation}\label{adjoint}
   v_t+v_{xx}+2v_{xy}+v_{yy}=0.
 \end{equation} Multiplying equations \eqref{sol1} and \eqref{adjoint} by $u$ and $v$ respectively and then adding the resulting equations we obtain an equation that can be written in divergence form:
 \begin{equation}\label{divform}
   (uv)_t+(uv_x-vu_x+uv_y-vu_y)_x+(uv_y-vu_y+uv_x-vu_x)_y=0. \end{equation}
Choosing the particular solution $v=e^{-i\kappa x-i\lambda y+(\kappa+\lambda)^2t}$ and employing the divergence theorem in $(x,y,\tau)\in \mathbb{R}\times \mathbb{R}_+\times (0,t)$, we get contributions only from the planes $\tau=0,x\in\mathbb{R},y\in\mathbb{R}_+;\tau=t,x\in\mathbb{R},y\in\mathbb{R}_+;x\in\mathbb{R},\tau\in[0,t],y=0.$ We find
\begin{equation}\label{altGR}
\begin{split}
  \left(\int_{-\infty}^{\infty}\int_{0}^{\infty}e^{-i\kappa x-iy\lambda}u(x,y,t)dxdy\right)e^{(\kappa+\lambda)^2t}=\int_{-\infty}^{\infty}\int_{0}^{\infty}e^{-i\kappa x-iy\lambda}u(x,y,0)dxdy\\
  -\int_0^t\int_{-\infty}^{\infty}e^{-i\kappa x+(\kappa+\lambda)^2\tau}u(x,0,\tau)d\tau dx-i(\lambda+2\kappa)\int_0^t\int_{-\infty}^{\infty}e^{-i\kappa x+(\kappa+\lambda)^2\tau}u_y(x,0,\tau)d\tau dx,
  \end{split}
\end{equation} which is equation \eqref{sol1}; the last two terms of the right hand side of this equation arise from the terms
$$uv_y-vu_y+uv_x-vu_x=e^{-i\kappa x-i\lambda y+(\kappa+\lambda)^2t}[-i(\lambda+\kappa)u-u_y-u_x],$$ where the last term in the above bracket can be replaced with $-i\kappa u$ using integration by parts.
\end{rem}
Inverting the global relation \eqref{sol1}, we find that $u$ must satisfy
\begin{equation}\label{sol2}
  u(x,y,t)=\frac{1}{(2\pi)^2}\int_{-\infty}^\infty\int_{-\infty}^{\infty}E(\kappa,\lambda;x,y,t) (\widehat{u_0}(\kappa,\lambda)-\tilde{g}(\kappa,\lambda;t))d\lambda d\kappa ,
\end{equation} where
\begin{equation*}
  E(\kappa,\lambda;x,y,t):=\exp(-(\kappa+\lambda)^2t+i\kappa x+i\lambda y)
\end{equation*} and
\begin{equation*}
  \tilde{g}(\kappa,\lambda;t):=\widetilde{g_1}(\kappa,(\kappa+\lambda)^2,t)+i(\lambda+2\kappa) \widetilde{g_0}(\kappa,(\kappa+\lambda)^2,t).
\end{equation*}
Note that only the Dirichlet boundary value is prescribed in the canonical model \eqref{maineq1}-\eqref{init1}. Thus, the $t$-transform $\widetilde{g_1}$ of $g_1$, which appears at the right hand side of \eqref{sol2}, is unknown.

In order to eliminate this unknown function from \eqref{sol2}, we will first deform the contour of integration for the $\lambda$ variable from the real line to a more suitable contour denoted $\partial D_\kappa^+$ in the upper half complex $\lambda$-plane. More precisely, we first define the family of complex domains
\begin{equation*}
  D_\kappa := \{\lambda\in\mathbb{C}\,|\,\Re(\kappa+\lambda)^2< 0\} \text{ (see Figure \ref{dkappa0})},
\end{equation*} parameterized with respect to $\kappa\in \mathbb{R}$. The part of $D_\kappa$ in the upper half complex $\lambda$-plane will be referred to as $D_\kappa^+$, i.e., $D_\kappa^+:=D_\kappa\cap \mathbb{C}_+$.  The definition of $D_\kappa$ is motivated by the fact that in $\mathbb{C}\setminus {D_\kappa}$, the term $\exp(-(\kappa+\lambda)^2(t-\tau))$ is bounded. Therefore, the integrand $E\tilde{g}$ is analytic and decays as $\lambda\rightarrow \infty$ for $\lambda\in \mathbb{C}\setminus D_\kappa^+$.  By the standard complex analytic arguments used in the Fokas method, namely by appealing to Cauchy's theorem and Jordan's lemma, the integral $\int_{-\infty}^\infty E\tilde{g}d\lambda$ can be replaced with $\int_{\partial D_\kappa^+} E\tilde{g}d\lambda$.  Therefore, we can rewrite \eqref{sol2} in the form
\begin{equation}\label{sol3}
\begin{split}
    u(x,y,t)&=\frac{1}{(2\pi)^2}\int_{-\infty}^\infty\int_{-\infty}^{\infty}E(\kappa,\lambda;x,y,t) \widehat{u_0}(\kappa,\lambda)d\lambda d\kappa\\
    & -\frac{1}{(2\pi)^2}\int_{-\infty}^\infty\int_{\partial D_\kappa^+}E(\kappa,\lambda;x,y,t) \tilde{g}(\kappa,\lambda,t)d\lambda d\kappa.
\end{split}
\end{equation}
\begin{figure}
  \centering
  \includegraphics[width=5cm]{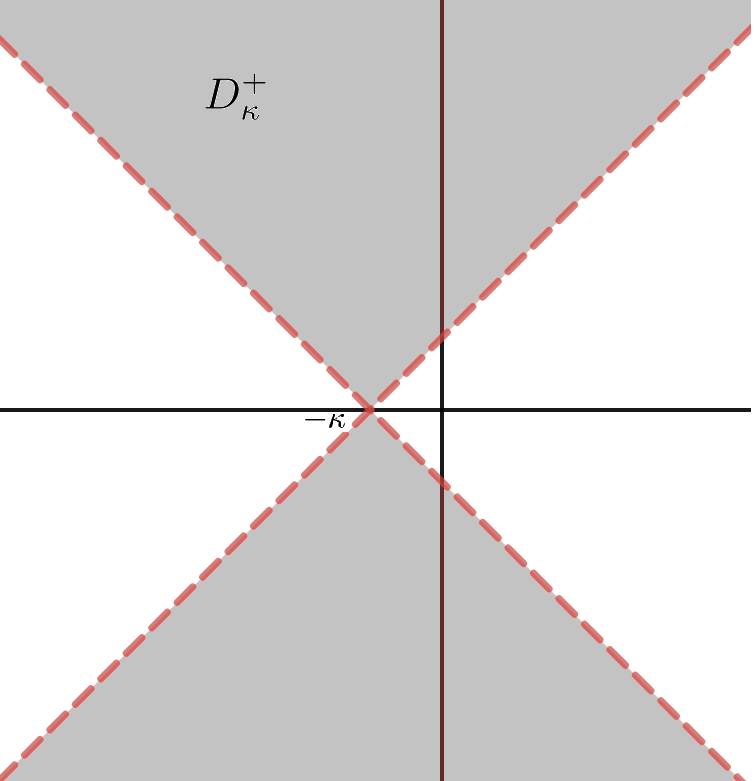}
  \caption{The shaded region shows $D_\kappa$ while the shaded region in the upper complex half plane shows $D_\kappa^+$}\label{dkappa0}
\end{figure}
The orientation of the contour $\partial D_\kappa^+$ is taken in such a way that $D_\kappa^+$ is to the left of $\partial D_\kappa^+$.

The second step is to eliminate the unknown $t$-transform $\tilde{g}_1$ by using the Dirichlet-to-Neumann-like map given below: replacing in the global relation \eqref{sol1} $\lambda$ with $-\lambda-2\kappa,$ we find
\begin{equation}\label{sol4}
\begin{split}
  &\widetilde{g_1}(\kappa,(\kappa+\lambda)^2,t)=-\hat{u}(\kappa,-(\lambda+2\kappa),t)\exp((\kappa+\lambda)^2t)\\
  &+ \widehat{u_0}(\kappa,-(\lambda+2\kappa))+i\lambda \widetilde{g_0}(\kappa,(\kappa+\lambda)^2,t),\quad \kappa\in \mathbb{R}, \Im\lambda\ge 0.
  \end{split}
\end{equation}
Note that \eqref{sol4} is valid in the upper half complex plane because the condition $\Im\lambda\le 0$ in the global relation \eqref{sol1} gives rise to the condition $\Im\lambda\ge 0$ under the mapping $\lambda \mapsto -\lambda-2\kappa$. This mapping is obtained by analysing the invariance properties of the secondary spectral input $(\kappa+\lambda)^2$ of the $t$-transforms of the Fourier transform of boundary values.  Remarkably, Cauchy's theorem and Jordan's lemma imply that the contribution of the term that involves $\hat{u}(\kappa,-(\lambda+2\kappa),t)$ to the integral over $\partial D_\kappa^+$ vanishes. Therefore, one obtains the following solution representation formula, whose right hand side involves only the given initial and boundary values:
\begin{equation}\label{solrep}
\begin{split}
    u(x,y,t)=&\frac{1}{(2\pi)^2}\int_{-\infty}^\infty\int_{-\infty}^{\infty}E(\kappa,\lambda;x,y,t) \widehat{u_0}(\kappa,\lambda)d\lambda d\kappa\\
    & -\frac{1}{(2\pi)^2}\int_{-\infty}^\infty\int_{\partial D_\kappa^+}E(\kappa,\lambda;x,y,t)\widehat{u_0}(\kappa,-(\lambda+2\kappa))d\lambda d\kappa\\
    &-\frac{i}{2\pi^2}\int_{-\infty}^\infty\int_{\partial D_\kappa^+}(\lambda+\kappa)E(\kappa,\lambda;x,y,t) \widetilde{g_0}(\kappa,(\kappa+\lambda)^2,t)d\lambda d\kappa.
\end{split}
\end{equation}
As usual $\widetilde{g_0}(\kappa,(\kappa+\lambda)^2,t)$ in \eqref{solrep} can be replaced with $\widetilde{g_0}(\kappa,(\kappa+\lambda)^2,T)$ because the contribution of $$(\lambda+\kappa)E(\kappa,\lambda;x,y,t) \int_t^T\exp((\kappa+\lambda)^2 \tau)\reallywidehat[x]{g_0}(\kappa,\tau)d\tau$$ to the integral over $\partial D_\kappa^+$ vanishes.

\section{Fourth order heat equation}
In this section, we obtain the representation formula for the solution of the initial boundary value problem for the fourth order heat equation posed on the half space.  In particular, we consider the following canonical model with Dirichlet-Neumann boundary conditions:
 \begin{align}
     \label{maineq2}u_t+\Delta^2 u&=0,\quad (x,y,t)\in\mathbb{R}\times \mathbb{R}_+\times (0,T),\\
     \label{bc2d}u(x,0,t)&=g_0(x,t), \quad (x,t)\in\mathbb{R}\times(0,T),\\
     \label{bc2n}u_y(x,0,t)&=g_1(x,t), \quad (x,t)\in\mathbb{R}\times(0,T),\\
     \label{init2}u(x,y,0)&=u_0(x,y), \quad \quad (x,y)\in\mathbb{R}\times\mathbb{R}_+.
     \end{align}
 In \eqref{maineq2}, the differential operator $\Delta^2$ is given by \begin{equation}\label{defdeltasq}\Delta^2 u = \partial_x^4u+2\partial_x^2\partial_y^2u+\partial_y^4u.\end{equation} Sometimes, we will write $\Delta_{x,y}$ to emphasize the fact that derivatives are taken with respect to variables $x$ and $y$.

We first note that there is no linear change of variables which eliminates the cross term $2u_{xxyy}$ from the definition of $\Delta^2$.
\begin{lem}
  There is no linear change of variables in the form $\tilde{x}=ax+by, \tilde{y}=cx+dy$ with $ad-bc\neq 0$ for which $\Delta_{x,y}^2u=(\partial_{\tilde{x}}^4+\partial_{\tilde{y}}^4)u$.
\end{lem}
\begin{proof}
  Assume the contrary, namely that a linear change of variables as in the statement of the lemma does exist.  Then, after some calculations it follows that
  \begin{equation*}\label{dercalc}
  \begin{split}
    \Delta_{x,y}^2u=&(a^2+b^2)^2\partial_{\tilde{x}}^4u+4(ac+bd)(a^2+b^2)\partial_{\tilde{x}}^3\partial_{\tilde{y}}u\\
    &+2(3a^2c^2+3b^2d^2+a^2d^2+b^2c^2+4abcd)\partial_{\tilde{x}}^2\partial_{\tilde{y}}^2u\\
    &+4(ac+bd)(c^2+d^2)\partial_{\tilde{x}}\partial_{\tilde{y}}^3u+(c^2+d^2)^2\partial_{\tilde{y}}^4u.\end{split}
  \end{equation*} Since  $ad-bc\neq 0$, we must have $ac=-bd$, but then
  $$3a^2c^2+3b^2d^2+a^2d^2+b^2c^2+4abcd=2b^2d^2+a^2d^2+b^2c^2=0.$$ Hence $ad=bc=0$, which is a contradiction.  Therefore, there is no linear change of variables with nonzero determinant as in the statement of the lemma.
\end{proof}
 Taking the half space Fourier transform of \eqref{maineq2}, after some calculations we obtain
 \begin{equation}\label{FTofMaineq2}
\begin{split}
\hat{u}_t(\kappa,\lambda,t) = -(\kappa^2+\lambda^2)^2\hat{u}(\kappa,\lambda,t)+\reallywidehat[x]{{g}_3}(\kappa,t)+i\lambda\reallywidehat[x]{{g}_2}(\kappa,t)\\
-(2\kappa^2+\lambda^2)\reallywidehat[x]{{g}_1}(\kappa,t)-i\lambda(2\kappa^2+\lambda^2)\reallywidehat[x]{{g}_0}(\kappa,t),\quad \kappa\in \mathbb{R}, \Im\lambda\le 0,
\end{split}
\end{equation} where $g_j(x,t):=\partial_y^ju(x,0,t)$ and $\reallywidehat[x]{{g}_j}$ is the classical Fourier transform of $g_j$ with respect to the $x$ variable (see \eqref{gjhat}). Integrating \eqref{FTofMaineq2} and using the condition $\hat{u}(\kappa,\lambda,0) = \widehat{u_0}(\kappa,\lambda)$,  we find the global relation
\begin{equation}\label{glob4}
\begin{split}
  \hat{u}(\kappa,\lambda,t)\exp((\kappa^2+\lambda^2)^2t)= \widehat{u_0}(\kappa,\lambda)+\widetilde{{g}_3}(\kappa,(\kappa^2+\lambda^2)^2,t)+i\lambda\widetilde{{g}_2}(\kappa,(\kappa^2+\lambda^2)^2,t)\\
-(2\kappa^2+\lambda^2)\widetilde{{g}_1}(\kappa,(\kappa^2+\lambda^2)^2,t)-i\lambda(2\kappa^2+\lambda^2)\widetilde{{g}_0}(\kappa,(\kappa^2+\lambda^2)^2,t),\quad \kappa\in \mathbb{R}, \Im\lambda\le 0.
  \end{split}
\end{equation}
\begin{rem}
In analogy with equation \eqref{altGR}, the global relation \eqref{glob4} can also be obtained by considering the adjoint to \eqref{maineq2}, namely the equation \begin{equation}\label{adjeq2}
           v_t-\Delta^2v=0.
         \end{equation}
Multiplying equation \eqref{maineq2} and \eqref{adjeq2} by $u$ and $v$ respectively, and then adding the resulting equations we find an equation that can be written in the divergence form:
\begin{equation}\label{divform2}
  \begin{split}
     (uv)_t+(vu_{xxx}-uv_{xxx}+vu_{xyy}-uv_{xyy}+u_xv_{xx}-v_xu_{xx}+u_yv_{xy}-v_yu_{xy})_x\\
     +(vu_{yyy}-uv_{yyy}+vu_{yxx}-uv_{yxx}+u_yv_{yy}-v_yu_{yy}+u_xv_{xy}-v_xu_{xy})_y=0.
  \end{split}
\end{equation}
Choosing the particular solution $v=\exp(-i\kappa x-i\lambda y+(\kappa+\lambda)^2t)$ and using integration by parts, the second parenthesis of the above equation gives rise to the term
\begin{equation*}
  \begin{split}
e^{-i\kappa x-i\lambda y+(\kappa+\lambda)^2t}[u_{yyy}-(-i\lambda)^3u+(-i\kappa)^2u_y-(-i\lambda)(-i\kappa)^2u+(-i\lambda)^2u_y\\
-(-i\lambda)u_{yy}-(-i\lambda)(-i\kappa)^2u+(-i\kappa)^2u_y]\\
=e^{-i\kappa x-i\lambda y+(\kappa+\lambda)^2t}[u_{yyy}+i\lambda u_{yy}-(2\kappa^2+\lambda^2)u_y-i\lambda(2\kappa^2+\lambda^2)u].
  \end{split}
\end{equation*}
Hence, following the same steps as those used for the derivation of \eqref{altGR} we find \eqref{FTofMaineq2}.
\end{rem}
Inverting the global relation \eqref{glob4}, we find that $u$ must satisfy
\begin{equation}\label{sol2b}
  u(x,y,t)=\frac{1}{(2\pi)^2}\int_{-\infty}^\infty\int_{-\infty}^{\infty}E(\kappa,\lambda;x,y,t) (\widehat{u_0}(\kappa,\lambda)+\tilde{g}(\kappa,\lambda;t))d\lambda d\kappa ,
\end{equation} where
\begin{equation*}
  E(\kappa,\lambda;x,y,t):=\exp(-(\kappa^2+\lambda^2)^2t+i\kappa x+i\lambda y)
\end{equation*} and
\begin{equation}\label{gtilde}
\begin{split}
 \tilde{g}(\kappa,\lambda;t):=\widetilde{{g}_3}(\kappa,(\kappa^2+\lambda^2)^2,t)+i\lambda\widetilde{{g}_2}(\kappa,(\kappa^2+\lambda^2)^2,t)\\
-(2\kappa^2+\lambda^2)\widetilde{{g}_1}(\kappa,(\kappa^2+\lambda^2)^2,t)-i\lambda(2\kappa^2+\lambda^2)\widetilde{{g}_0}(\kappa,(\kappa^2+\lambda^2)^2,t),
\end{split}
\end{equation} where $\widetilde{g_j}$ is defined in \eqref{ttransform}.   Note that for the canonical problem defined by \eqref{maineq2}-\eqref{init2} the second and third order boundary traces, that is $g_2$ and $g_3$, are unknown functions.  In order to eliminate the $t-$transforms of $g_2$ and $g_3$ from the right hand side of \eqref{sol2b}, we again consider a family of complex regions $D_\kappa$ in the complex $\lambda$-plane parameterized by $\kappa\in \mathbb{R}$. We define
\begin{equation*}
  D_\kappa := \{\lambda\in\mathbb{C}\,|\,\Re(\kappa^2+\lambda^2)^2< 0\}, D_\kappa^+=D_\kappa\cap \mathbb{C}_+.
\end{equation*} The shaded region in Figure \ref{dkappa} is an illustration of $D_\kappa$. The graph of $D_\kappa$ cuts the imaginary axis at the points $\mp i\kappa$.

\begin{figure}
  \centering
  \includegraphics[width=8cm]{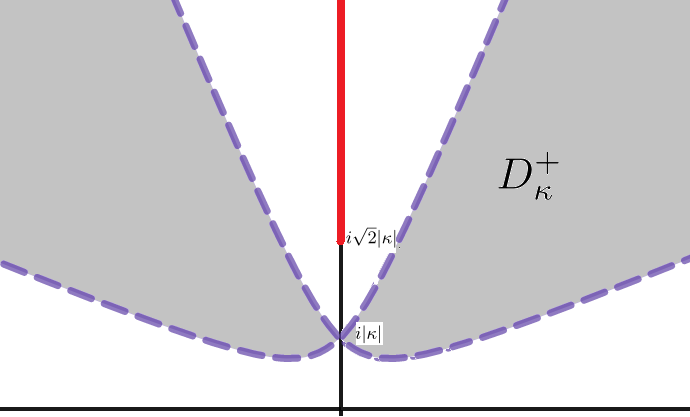}
  \caption{The shaded region shows $D_\kappa^+$ and the red line segment shows the branch cut of $\nu_\kappa$}\label{dkappa}
\end{figure}
Using the same complex analytic arguments as in the previous section, we can rewrite \eqref{sol2b} in the form
\begin{equation}\label{sol3b}
\begin{split}
    u(x,y,t)&=\frac{1}{(2\pi)^2}\int_{-\infty}^\infty\int_{-\infty}^{\infty}E(\kappa,\lambda;x,y,t) \widehat{u_0}(\kappa,\lambda)d\lambda d\kappa\\
    & +\frac{1}{(2\pi)^2}\int_{-\infty}^\infty\int_{\partial D_\kappa^+}E(\kappa,\lambda;x,y,t) \tilde{g}(\kappa,\lambda,t)d\lambda d\kappa.
\end{split}
\end{equation}
In order to implement the next step of the UTM we must find, for given $\kappa\in\mathbb{R}$, nontrivial maps $\lambda\mapsto \nu_\kappa(\lambda)$ which keep the spectral input $(\kappa^2+\lambda^2)^2$ invariant. Thus
\begin{equation}\label{invmap}
  (\kappa^2+\lambda^2)^2=(\kappa^2+\nu_\kappa^2(\lambda))^2:\,\,\,\,\text{(i) } \nu_\kappa(\lambda)=-\lambda \text{ or } \text{(ii) } \nu_\kappa^2(\lambda)=-\lambda^2-2\kappa^2.
\end{equation}  The transformation $\lambda\mapsto -\lambda$ yields the following global relation, which is valid in the upper half complex $\lambda-$plane for any $\kappa\in\mathbb{R}$:
\begin{equation}\label{glob5}
\begin{split}
  \hat{u}(\kappa,-\lambda,t)\exp((\kappa^2+\lambda^2)^2t)= \widehat{u_0}(\kappa,-\lambda)+\widetilde{{g}_3}(\kappa,(\kappa^2+\lambda^2)^2,t)-i\lambda\widetilde{{g}_2}(\kappa,(\kappa^2+\lambda^2)^2,t)\\
-(2\kappa^2+\lambda^2)\widetilde{{g}_1}(\kappa,(\kappa^2+\lambda^2)^2,t)+i\lambda(2\kappa^2+\lambda^2)\widetilde{{g}_0}(\kappa,(\kappa^2+\lambda^2)^2,t),\quad \kappa\in \mathbb{R}, \Im\lambda\ge 0.
  \end{split}
\end{equation} If $\nu_\kappa$ is an invariance map satisfying \eqref{invmap}-(ii), then using \eqref{glob5} we also find the global relation
\begin{equation}\label{glob6}
\begin{split}
  &\hat{u}(\kappa,-\nu_\kappa(\lambda),t)\exp((\kappa^2+\lambda^2)^2t)= \widehat{u_0}(\kappa,-\nu_\kappa(\lambda))+\widetilde{{g}_3}(\kappa,(\kappa^2+\lambda^2)^2,t)\\
  &-i\nu_\kappa(\lambda)\widetilde{{g}_2}(\kappa,(\kappa^2+\lambda^2)^2,t)
+\lambda^2\widetilde{{g}_1}(\kappa,(\kappa^2+\lambda^2)^2,t)-i\lambda^2\nu_\kappa(\lambda)\widetilde{{g}_0}(\kappa,(\kappa^2+\lambda^2)^2,t),\\
&\quad \kappa\in \mathbb{R}, \Im\nu_\kappa(\lambda)\ge 0.
  \end{split}
\end{equation}
For $\Im\lambda\ge 0$ and $\Im\nu_\kappa(\lambda)\ge 0$, both \eqref{glob5} and \eqref{glob6} are valid, and therefore we can use \eqref{glob5} and \eqref{glob6} to eliminate the unknown $t-$transforms in the second integral of the right hand side of \eqref{sol3b}. Solving \eqref{glob5} and \eqref{glob6} for the unknowns $\widetilde{{g}_3}$ and $\widetilde{{g}_2}$, we obtain the following equations, which are valid for $\kappa\in\mathbb{R}$, $\Im\lambda\ge 0$, $\Im\nu_\kappa(\lambda)\ge 0$, and $\nu_\kappa(\lambda)\neq \lambda$:
\begin{equation}\label{g2tilde}
\begin{split}
   \widetilde{{g}_2}(\kappa,(\kappa^2+\lambda^2)^2,t) &  =-\frac{i}{\lambda-\nu_\kappa(\lambda)}\left[\widehat{u_0}(\kappa,-\lambda)-\widehat{u_0}(\kappa,-\nu_\kappa(\lambda))\right] \\
     & +\frac{i}{\lambda-\nu_\kappa(\lambda)}\exp((\kappa^2+\lambda^2)^2t)\left[\hat{u}(\kappa,-\lambda,t)-\hat{u}(\kappa,-\nu_\kappa(\lambda),t)\right]\\
     & +i(\nu_\kappa(\lambda)+\lambda)\widetilde{{g}_1}(\kappa,(\kappa^2+\lambda^2)^2,t)
     +\lambda\nu_\kappa(\lambda)\widetilde{{g}_0}(\kappa,(\kappa^2+\lambda^2)^2,t)
\end{split}
\end{equation} and
\begin{equation}\label{g3tilde}
\begin{split}
   \widetilde{{g}_3}(\kappa,(\kappa^2+\lambda^2)^2,t) &  =-\frac{1}{\nu_\kappa(\lambda)-\lambda}\left[\nu_\kappa(\lambda)\widehat{u_0}(\kappa,-\lambda)-\lambda\widehat{u_0}(\kappa,-\nu_\kappa(\lambda))\right] \\
     & -\frac{1}{\nu_\kappa(\lambda)-\lambda}\exp((\kappa^2+\lambda^2)^2t)\left[\lambda\hat{u}(\kappa,-\nu_\kappa(\lambda),t)-\nu_\kappa(\lambda)\hat{u}(\kappa,-\lambda,t)\right]\\
     & -(\lambda^2+\lambda\nu_\kappa(\lambda)+\nu_\kappa^2(\lambda))\widetilde{{g}_1}(\kappa,(\kappa^2+\lambda^2)^2,t)\\
     & +i\lambda\nu_k(\lambda)(\nu_k+\lambda)\widetilde{{g}_0}(\kappa,(\kappa^2+\lambda^2)^2,t).
\end{split}
\end{equation}
Using \eqref{g2tilde} and \eqref{g3tilde} in \eqref{gtilde}, we find that in the upper  half complex $\lambda-$plane, $\tilde{g}$ takes the form
\begin{equation}\label{gtildenew}
\begin{split}
 \tilde{g}(\kappa,\lambda;t):=-\left[\frac{\nu_\kappa(\lambda)+\lambda}{\nu_\kappa(\lambda)-\lambda}\right]\widehat{u_0}(\kappa,-\lambda)
 +\frac{2\lambda}{\nu_\kappa(\lambda)-\lambda}\widehat{u_0}(\kappa,-\nu_\kappa(\lambda))\\
 +\left[\frac{\nu_\kappa(\lambda)+\lambda}{\nu_\kappa(\lambda)-\lambda}\right]\exp((\kappa^2+\lambda^2)^2t)\hat{u}(\kappa,-\lambda,t)
 -\frac{2\lambda}{\nu_\kappa(\lambda)-\lambda}\exp((\kappa^2+\lambda^2)^2t)\hat{u}(\kappa,-\nu_\kappa(\lambda),t)\\
 -2\lambda[\lambda+\nu_\kappa(\lambda)]\widetilde{{g}_1}(\kappa,(\kappa^2+\lambda^2)^2,t) +2i\lambda\nu_\kappa(\lambda)[\nu_\kappa(\lambda)+\lambda]\widetilde{{g}_0}(\kappa,(\kappa^2+\lambda^2)^2,t).
\end{split}
\end{equation}
Given $\kappa\in \mathbb{R}$, the condition $\nu_\kappa(\lambda)= \lambda$ holds in the upper half complex $\lambda-$plane if $\lambda=i|\kappa|$. However, it is clear from the above form of $\tilde{g}$ that $\lambda=i|\kappa|$ is a removable singularity.

In order that $\nu_\kappa$ has the desired properties needed in the above derivations, it is necessary to choose an appropriate branch of the square root.  This branch is given by $\nu_\kappa(\lambda)=\sqrt{-\lambda^2-2\kappa^2}$, where for $z\in\mathbb{C}$, we define $\displaystyle\sqrt{z}=|z|^{\frac{1}{2}}e^{i\frac{\arg z}{2}}$ with  $\arg z\in [0,2\pi)$.  This function is analytic in the complement of the set $S_{\kappa}=\{\lambda\in \mathbb{C}\,|\,\Re \lambda= 0,|\Im \lambda|\ge \sqrt{2}|\kappa|\}$, hence $\nu_\kappa$ is indeed analytic since the set $S_{\kappa}$ is  outside and strictly away from $\overline{D_\kappa^+}$ (see Figure \ref{dkappa}).  Therefore, the contribution of the terms involving $\hat{u}(\kappa,-\lambda,t)$ and $\hat{u}(\kappa,-\nu_\kappa(\lambda),t)$ to the integral \eqref{sol3b} vanish on $\partial D_\kappa^+$.  Thus, the solution of \eqref{maineq2}-\eqref{init2} can be represented in the form
\begin{equation}\label{sol3bnew}
\begin{split}
    u(x,y,t)&=\frac{1}{(2\pi)^2}\int_{-\infty}^\infty\int_{-\infty}^{\infty}E(\kappa,\lambda;x,y,t) \widehat{u_0}(\kappa,\lambda)d\lambda d\kappa\\
    & +\frac{1}{(2\pi)^2}\int_{-\infty}^\infty\int_{\partial D_\kappa^+}E(\kappa,\lambda;x,y,t) G(\kappa,\lambda,t)d\lambda d\kappa,
\end{split}
\end{equation}
where
\begin{equation}\label{Gdef}
\begin{split}
 G(\kappa,\lambda;t):=-\left[\frac{\nu_\kappa(\lambda)+\lambda}{\nu_\kappa(\lambda)-\lambda}\right]\widehat{u_0}(\kappa,-\lambda)
 +\frac{2\lambda}{\nu_\kappa(\lambda)-\lambda}\widehat{u_0}(\kappa,-\nu_\kappa(\lambda))\\
 -2\lambda[\lambda+\nu_\kappa(\lambda)]\widetilde{{g}_1}(\kappa,(\kappa^2+\lambda^2)^2,t) +2i\lambda\nu_\kappa(\lambda)[\nu_\kappa(\lambda)+\lambda]\widetilde{{g}_0}(\kappa,(\kappa^2+\lambda^2)^2,t)
\end{split}
\end{equation} and $\nu_\kappa(\lambda)=\sqrt{-\lambda^2-2\kappa^2}$.

\section{Fourth order Schrödinger equation}
We consider the fourth order Schrödinger equation:
 \begin{align}
     \label{maineq3}iu_t+\Delta^2 u&=0,\quad (x,y,t)\in\mathbb{R}\times \mathbb{R}_+\times (0,T),\\
     \label{bc3d}u(x,0,t)&=g_0(x,t), \quad (x,t)\in\mathbb{R}\times(0,T),\\
     \label{bc3n}u_y(x,0,t)&=g_1(x,t), \quad (x,t)\in\mathbb{R}\times(0,T),\\
     \label{init3}u(x,y,0)&=u_0(x,y), \quad \quad (x,y)\in\mathbb{R}\times\mathbb{R}_+,
     \end{align} where $u$ is a complex valued function and $\Delta^2$ is as in \eqref{defdeltasq}.
 \subsection{Global relations}
 We take the half space Fourier transform of \eqref{maineq3} and obtain the global relation
\begin{equation}\label{GR}
\begin{split}
\hat{u}(\kappa,\lambda,t)\exp({-i(\kappa^2+\lambda^2)^2t})=\widehat{u_0}(\kappa,\lambda)+\tilde{g}(\kappa,\lambda;t),\kappa\in\mathbb{R}, \Im\lambda\le 0,
\end{split}
\end{equation} where
\begin{equation*}\label{gtilde3}
\begin{split}
 \tilde{g}(\kappa,\lambda;t):=-i\tilde{g}_3(\kappa,-i(\kappa^2+\lambda^2)^2,t)+\lambda\tilde{g}_2(\kappa,-i(\kappa^2+\lambda^2)^2,t)\\
+i(\lambda^2+2\kappa^2)\tilde{g}_1(\kappa,-i(\kappa^2+\lambda^2)^2,t)-\lambda(\lambda^2+2\kappa^2)\tilde{g}_0(\kappa,-i(\kappa^2+\lambda^2)^2,t).
\end{split}
\end{equation*}
\begin{rem}
An alternative derivation is based on the function $$v=\exp(-i\kappa x-i\lambda y-i(\kappa^2+\lambda^2)^2t),$$ which is particular solution of the adjoint equation $$iv_t-\Delta^2v=0.$$
\end{rem}
We define
\begin{equation*}
  D_\kappa := \{\lambda\in\mathbb{C}\,|\,\Re[i(\kappa^2+\lambda^2)^2]> 0\}, D_\kappa^+=D_\kappa\cap \mathbb{C}_+.
\end{equation*} The shaded region in Figure \ref{graph} is an illustration of $D_\kappa^+$. The graph of $D_\kappa^+$ cuts the imaginary axis at the point $i|\kappa|$.
\begin{figure}
  \centering
  \includegraphics[width=8cm]{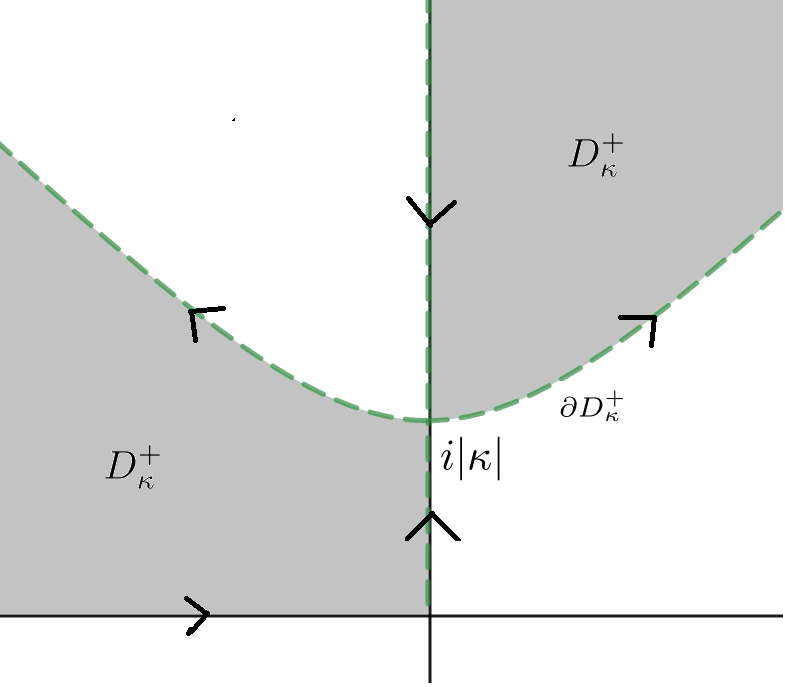}
  \caption{The shaded region shows $D_\kappa^+$}\label{graph}
\end{figure}
Taking the inverse Fourier transform  in \eqref{GR} and deforming the contour of $\lambda$ from the real line to the boundary of $D_\kappa^+$, we obtain
\begin{equation}\label{int2}
\begin{split}
  u(x,y,t)=\frac{1}{(2\pi)^2}\int_{-\infty}^\infty\int_{-\infty}^\infty E(\kappa,\lambda;x,y,t)\widehat{u_0}(\kappa,\lambda)d\lambda d\kappa\\
  \frac{1}{(2\pi)^2}\int_{-\infty}^\infty\int_{\partial D_\kappa^+}E(\kappa,\lambda;x,y,t) \tilde{g}(\kappa,\lambda;t)d\lambda d\kappa ,
  \end{split}
\end{equation} where
\begin{equation*}
  E(\kappa,\lambda;x,y,t):=\exp(i(\kappa^2+\lambda^2)^2t+i\kappa x+i\lambda y).
\end{equation*}
Replacing $\lambda$ with $-\lambda$ in \eqref{GR}, we find a global relation which is valid for $\Im \lambda\ge 0$:
\begin{equation}\label{GR2}
\begin{split}
&\hat{u}(\kappa,-\lambda,t)\exp({-i(\kappa^2+\lambda^2)^2t})=\widehat{u_0}(\kappa,-\lambda)\\
&-i\tilde{g}_3(\kappa,-i(\kappa^2+\lambda^2)^2,t)-\lambda\tilde{g}_2(\kappa,-i(\kappa^2+\lambda^2)^2,t)\\
&+i(\lambda^2+2\kappa^2)\tilde{g}_1(\kappa,-i(\kappa^2+\lambda^2)^2,t)+\lambda(\lambda^2+2\kappa^2)\tilde{g}_0(\kappa,-i(\kappa^2+\lambda^2)^2,t).
\end{split}
\end{equation} Furthermore, if we change $\lambda$ in \eqref{GR2} with $\nu_\kappa$ satisfying
\begin{equation}\label{nuu}
\nu_\kappa^2=-\lambda^2-2\kappa^2,
\end{equation} we obtain a new global relation which is valid for $\Im\lambda\ge 0$ and $\Im \nu_\kappa\ge 0$:
\begin{equation}\label{GR3}
\begin{split}
&\hat{u}(\kappa,-\nu_\kappa,t)\exp({-i(\kappa^2+\lambda^2)^2t})=\widehat{u_0}(\kappa,-\nu_\kappa)\\
&-i\tilde{g}_3(\kappa,-i(\kappa^2+\lambda^2)^2,t)-\nu_\kappa\tilde{g}_2(\kappa,-i(\kappa^2+\lambda^2)^2,t)\\
&-i\lambda^2\tilde{g}_1(\kappa,-i(\kappa^2+\lambda^2)^2,t)-\nu_\kappa\lambda^2\tilde{g}_0(\kappa,-i(\kappa^2+\lambda^2)^2,t).
\end{split}
\end{equation} Solving \eqref{GR2} and \eqref{GR3} for $\tilde{g}_2$ and $\tilde{g}_3$, we find
\begin{equation}\begin{split}
&\tilde{g}_2(\kappa,-i(\kappa^2+\lambda^2)^2,t)=\lambda \nu_\kappa \tilde{g}_0(\kappa,-i(\kappa^2+\lambda^2)^2,t)+i (\lambda+\nu_\kappa)\tilde{g}_1(\kappa,-i(\kappa^2+\lambda^2)^2,t)\\
&+\frac{\widehat{u_0}(\kappa,-\nu_\kappa)-\widehat{u_0}(\kappa,-\lambda)}{\nu_\kappa-\lambda}+\frac{\widehat{u}(\kappa,-\lambda,t)-\widehat{u}(\kappa,-\nu_\kappa,t)}{\nu_\kappa-\lambda}\exp({-i(\kappa^2+\lambda^2)^2t})
\end{split}
\end{equation} and
\begin{equation}
\begin{split}
&\tilde{g}_3(\kappa,-i(\kappa^2+\lambda^2)^2,t)=i\lambda\nu_\kappa(\lambda+\nu_\kappa)\tilde{g}_0(\kappa,-i(\kappa^2+\lambda^2)^2,t)\\
&-(\nu_\kappa^2+\nu_\kappa\lambda+\lambda^2)\tilde{g}_1(\kappa,-i(\kappa^2+\lambda^2)^2,t)+\frac{i\lambda \widehat{u_0}(\kappa,-\nu_\kappa)-i\nu_\kappa \widehat{u_0}(\kappa,-\lambda,t)}{\nu_\kappa-\lambda}\\
&+\frac{i\nu_\kappa \hat{u}(\kappa,-\lambda,t)-i\lambda \hat{u}(\kappa,-\nu_\kappa,t)}{\nu_\kappa-\lambda}\exp({-i(\kappa^2+\lambda^2)^2t}).
\end{split}
\end{equation} Using the above identities in \eqref{int2}, we obtain
\begin{equation}
\begin{split}\label{abcd}
u(x,y,t)=&\frac{1}{(2\pi)^2}\int_{-\infty}^\infty\int_{-\infty}^{\infty}E(\kappa,\lambda;x,y,t) \widehat{u_0}(\kappa,\lambda)d\lambda d\kappa\\
&+\frac{1}{(2\pi)^2}\int_{-\infty}^\infty\int_{\partial D_\kappa^+}E(\kappa,\lambda;x,y,t)G(\kappa,\lambda,t)d\lambda d\kappa\\
&+\frac{1}{(2\pi)^2}\int_{-\infty}^\infty\exp({i\kappa x})\int_{\partial D_\kappa^+}H(\lambda;\kappa,y,t)d\lambda d\kappa,
\end{split}
\end{equation} where
\begin{equation*}
\begin{split}
G(\kappa,\lambda,t):=\left[\frac{2\lambda}{\nu_\kappa-\lambda}\widehat{u_0}(\kappa,-\nu_\kappa)-\frac{\lambda+\nu_\kappa}{\nu_\kappa-\lambda}\widehat{u_0}(\kappa,-\lambda)\right]\\
+[2(\lambda^2\nu_\kappa+\lambda\nu_\kappa^2)\tilde{g}_0+2i(\lambda^2+\lambda\nu_\kappa)\tilde{g}_1]
\end{split}
\end{equation*} and
$$H(\lambda;\kappa,y,t):= \left[\frac{\lambda+\nu_\kappa}{\nu_\kappa-\lambda}\hat{u}(\kappa,-\lambda,t)-\frac{2\lambda}{\nu_\kappa-\lambda}\hat{u}(\kappa,-\nu_\kappa,t)\right]\exp({i\lambda y}). $$
\subsection{{Analyticity issues}}The point $\lambda = i|\kappa|$ is a removable singularity, thus if $\nu_\kappa$ was analytic function then $H(\lambda;\kappa,y,t)$ would be analytic in $D_\kappa^+$, and its contribution to the integral on $\partial D_\kappa^+$ would be zero. However, it is not clear whether there exists a map $\lambda\mapsto \nu_\kappa(\lambda)$ which both satisfies the invariance property \eqref{nuu} and the property (P) given below.

\vspace{3mm}
 (P)\,\, $\nu_{\kappa}$ is analytic on a finite family of open sets, each of which contains one connected component of $D_\kappa^+$.
\vspace{3mm}

We prove below in Lemma \ref{nonana} that such a map does not exist for our problem.
\begin{lem}\label{nonana}
  A function which satisfies the invariance property \eqref{nuu} can not be analytic in the neighbourhood of $\lambda=\pm i\sqrt{2}|\kappa|$ for $\kappa\neq 0$.
\end{lem}
\begin{proof}
  Assume that $\kappa \neq 0$. We give the proof only for $\lambda=i\sqrt{2}|\kappa|$; $\lambda=-i\sqrt{2}|\kappa|$ can be treated analogously. Let $S_\kappa^+$ be the intersection of $S_\kappa$, whose definition is given in Section 3, with the upper half complex $\lambda-$plane. Suppose that there is a function $\lambda\mapsto \nu_\kappa(\lambda)$ which is analytic in a neighborhood of $i\sqrt{2}|\kappa|$ denoted by $B_{\tilde{\epsilon}}(i\sqrt{2}|\kappa|)$ which also satisfies the invariance property \eqref{nuu}.  Let $\lambda_0$ be such that $S_\kappa^+\cap B_{\tilde{\epsilon}}(i\sqrt{2}|\kappa|)\ni \lambda_0\neq i\sqrt{2}|\kappa|$. Let $B_\epsilon(\lambda_0)$ be an open ball of radius $\epsilon$ centered at $\lambda_0$, and further $\epsilon$ is sufficiently small so that $i\sqrt{2}|\kappa|\notin B_\epsilon(\lambda_0)$. Consider the restriction map ${\nu_\kappa}|_{B_\epsilon(\lambda_0)}$. This function must be either of the form
  \[
	{\nu_\kappa}_{|_{B_\epsilon(\lambda_0)}} =
	\begin{cases}
	\nu^1_\kappa(\lambda) & \text{ $\Re\lambda>0$} \\
	-\nu^1_\kappa(\lambda) & \text{$\Re\lambda\leq 0$}
	\end{cases}
	\] or its negative, where $\nu^1_\kappa(\lambda)=\sqrt{-\lambda^2-2\kappa^2}$. We take a path in $B_{\tilde{\epsilon}}(i\sqrt{2}|\kappa|)$ starting from a point $A$ at the left side of $B_\epsilon(\lambda_0)$ and arriving at a point $B$ at the right side of  $B_\epsilon(\lambda_0)$ while avoiding $S^+_\kappa$ (see Figure \ref{yolY}). Then we apply analytic continuation to ${\nu_\kappa}_{|_{B_\epsilon(\lambda_0)}}$ starting from $A$ through the path. This would imply that $\nu_\kappa$ is equal to both $\nu^1_\kappa(\lambda)$ and $-\nu^1_\kappa(\lambda)$ at the point $B$, which is a contradiction.
\begin{figure}[h]
	\centering
	\includegraphics[width=6cm]{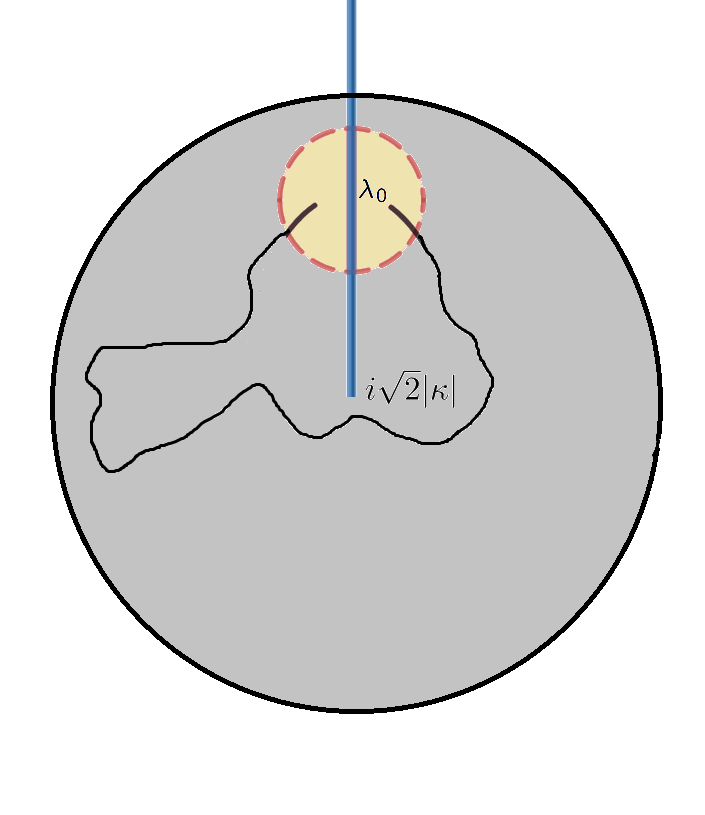}
	\caption{The (black) path is used for the analytic continuation argument in the proof of Lemma \ref{nonana}}\label{yolY}
\end{figure}
\end{proof}

Lemma \ref{nonana} shows that $\pm i\sqrt{2}|\kappa|$ are two branch points of any function $\nu_{\kappa}$ which satisfies \eqref{nuu}. Therefore, there is no map $\nu_{\kappa}$ which satisfies both \eqref{nuu} and the property (P). However, property (P) is not a necessary condition to say that the contribution of $H$ to the integral representation \eqref{abcd} is zero. Indeed, it would be enough that $\nu_{\kappa}$ were analytic on $\overline{D_\kappa
^+}\setminus\{i\sqrt{2}|\kappa|\}$ (see \eqref{inttH}). Fortunately, we can construct a single valued function $\nu_{\kappa}$ which is analytic on $\overline{D_\kappa}\setminus\{\pm i\sqrt{2}|\kappa|\}$ by taking an arbitrary branch cut consisting of disjoint curves starting from branch points $i\sqrt{2}|\kappa|$ and $-i\sqrt{2}|\kappa|$, and extending to infinity avoiding $\overline{D_\kappa}$ (except at the branch points). In order to be able to compute $\nu_{\kappa}(\lambda)$ for given $\lambda$ and make sure that it also satisfies $\Im \nu_{\kappa} \geq 0$ on $D_\kappa^+$, we will fix a suitable branch cut for it in the following way.

Let $s$ denote a (branch of the) square root function, $g(\lambda)=-\lambda^2-2\kappa^2$, and $\nu_\kappa=s\circ g$. Denoting the set of elements of the branch cut of a function $f$ by $W_f$, it is clear that $$W_{\nu_{\kappa}}=g^{-1}(W_s)= \sqrt{-W_s-2\kappa^2}\cup -\sqrt{-W_s-2\kappa^2},$$ where $\sqrt{\;\;}$ is the standard square root function. Then
\begin{equation}
\label{branch}
W^+_{\nu_{\kappa}}:=W_{\nu_{\kappa}}\cap \mathbb{C}^+=\sqrt{-W_s-2\kappa^2}.
\end{equation}
Note that if  $s$ is the standard square root function $\sqrt{\;\;}$, then $$W_s=\{z\in\mathbb{C}\;|\;\Re z\geq 0,\,\Im z= 0\},$$ and
$\nu_\kappa(\lambda)=\sqrt{-\lambda^2-2\kappa^2}$ satisfies the condition $\Im \nu_\kappa(\lambda)\ge 0$. However from \eqref{branch} we see that with this choice of $s$, upper part $W^+_{\nu_{\kappa}}$ of the corresponding branch cut, where $\nu_\kappa$ is not analytic, equals $S_{\kappa}^+$, which is a subset of $\partial D_\kappa^+$. Therefore, the standard square root is not the right branch to use for eliminating the integral of $H$ in \eqref{abcd}.

In order to deal with this issue we rotate the branch cut of the standard square root function a little bit. Namely, we fix a square root function $s$ as the map $z\mapsto \sqrt{z}^*$, where $\displaystyle \sqrt{z}^*:=|z|^{\frac{1}{2}}e^{i\frac{\arg z}{2}}$ with $\arg z\in [\epsilon,2\pi+\epsilon)$ for some fixed and sufficiently small $\epsilon>0$. Note that, now we have $$W_{s}=\{z\in\mathbb{C}\;|\;\Re z\geq 0,\,\Im z= \tan(\epsilon)\Re z\}$$ and from \eqref{branch} it is easy to see that $W^+_{\nu_{\kappa}}$ stays in the second quadrant. Moreover, as we show in the following lemma, for $\epsilon>0$ small enough, it is certain that $W_{\nu_\kappa}^+$ stays away from all  $\lambda\in\overline{D_\kappa^+}\setminus\{i\sqrt{2}\kappa\}$ and $\Im \nu_\kappa(\lambda)\ge 0$ for $\lambda\in D_\kappa^+,$ (see Figure \ref{stayaway}).
\begin{lem}\label{discont}
  For $\epsilon>0$ small enough, the function $\nu_\kappa(\lambda):=\sqrt{-\lambda^2-2\kappa^2}^*$ is analytic on $\overline{D_\kappa}\setminus\{\pm i\sqrt{2}|\kappa|\}$. Moreover, $\Im \nu_\kappa(\lambda)\ge 0$ for $\lambda\in D_\kappa.$
\end{lem}
\begin{proof}
  Let $W_{\nu_\kappa}$ be the set of points in the $\lambda-$plane where $\nu_\kappa$ is discontinuous. These are the points which satisfy $-\lambda^2-2\kappa^2=r_\lambda e^{i\epsilon}$ for some $r_\lambda\ge 0$.  For such points, we have
  \begin{eqnarray}
    \Re\lambda^2-\Im\lambda^2 &=& -r_\lambda\cos\epsilon-2\kappa^2, \\
    2\Re\lambda\Im\lambda &=& -r_\lambda\sin\epsilon.\label{alphasign}
  \end{eqnarray}
  Suppose $W_{\nu_\kappa}\ni \lambda\neq \pm i\sqrt{2}|\kappa|$, i.e., $r_\lambda\neq 0$. Set $\displaystyle\alpha_\lambda:=\frac{\Re \lambda}{\Im \lambda}$ and $\displaystyle\beta_{\epsilon,\lambda}:=2\cot(\epsilon)+\frac{4\kappa^2}{r_\lambda\sin{\epsilon}}$. Note that $\alpha_\lambda<0$ by \eqref{alphasign} and $\displaystyle\alpha_\lambda-\frac{1}{\alpha_\lambda}=\beta_{\epsilon,\lambda}>0.$  This implies $\alpha_\lambda>-\frac{1}{\beta_{\epsilon,\lambda}}>-\frac{\tan(\epsilon)}{2}$.  Therefore, we can make $|\alpha_\lambda|$ very small by taking $\epsilon>0$ small enough so that elements of $W_{\nu_\kappa}$ are strictly away from $\lambda\in\overline{D_\kappa}\setminus\{\pm i\sqrt{2}|\kappa|\}$.

  Moreover $\Im \sqrt{z}^* < 0$ only if $2\pi<\arg z <2\pi+\epsilon$. This translates into the region between $W_{\nu_\kappa}$ and $S_\kappa$ for $\nu_\kappa$ in complex $\lambda$-plane . Therefore the last statement of the lemma follows.
\end{proof}
\begin{figure}[h]
	\centering
	\includegraphics[width=8cm]{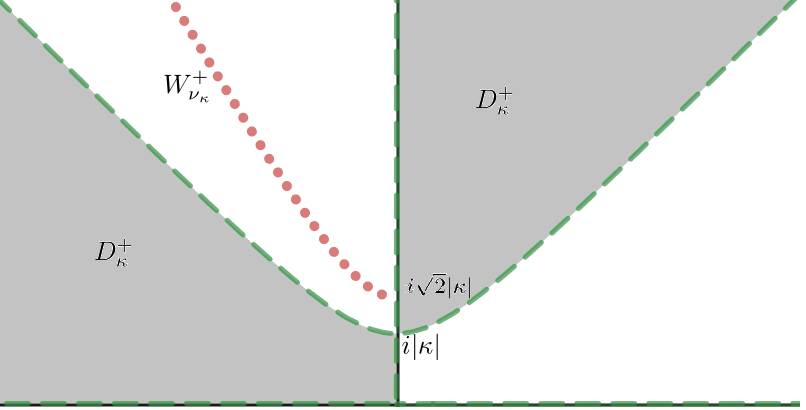}
	\caption{The shaded region shows $D_\kappa^+$ and the dashed red curve shows the branch cut}\label{stayaway}
\end{figure}
The next step in view of Lemma \ref{discont} is to take care of the branch point $\lambda=i\sqrt{2}|\kappa|$.  This can be achieved by deforming the contour to a small half circle around this point. More precisely, given $r>0$, let $B_{r}$ be the half disk defined by $$B_r\equiv\{\lambda\in\mathbb{C} \;|\; |\lambda-i\sqrt{2}\kappa|<r, \;\Re\lambda>0\}.$$
Then, the integral $\int_{\partial D^+}Hd\lambda$ can be written as
\begin{equation}
\label{inttH}
\int_{\partial D^+}Hd\lambda= \int_{\partial B_{r}}Hd\lambda+ \int_{\partial
(D^+\backslash B_{r})}Hd\lambda.
\end{equation}

 Since $H$ is analytic and bounded in $D^+\backslash B_{r}$ , using Cauchy's Theorem and Jordan's Lemma we find that the integral on $\partial (D^+\backslash B_{r})$ is zero. The function $\nu_\kappa(\lambda)$ is bounded even if it is not analytic about the point  $i\sqrt{2}|\kappa|$. $H$ is continuous with respect to $\lambda$ and $\nu_\kappa(\lambda)$, thus $H$ is also bounded on $B_{r}$. Therefore, the integral of $H$ on $\partial B_{r}$ goes to zero as $r$ goes to zero. Hence, $\int_{\partial D^+}Hd\lambda=0$. Thus, using \eqref{abcd}, we obtain
 \begin{equation}
\begin{split}\label{abcde}
u(x,y,t)=&\frac{1}{(2\pi)^2}\int_{-\infty}^\infty\int_{-\infty}^{\infty}E(\kappa,\lambda;x,y,t) \widehat{u_0}(\kappa,\lambda)d\lambda d\kappa\\
&+\frac{1}{(2\pi)^2}\int_{-\infty}^\infty\int_{\partial D^+}E(\kappa,\lambda;x,y,t)G(\kappa,\lambda,t)d\lambda d\kappa.
\end{split}
\end{equation}

\begin{rem}[Asymptotic form]The wellposedness analysis is generally easier if the solution formula is given on the boundary of the asymptotic form of $D_\kappa^+$, namely on $\partial D_{R}^+$, where $D_{R}^+=D_{\kappa,R}\cap\mathbb{C}_+$ and
$$D_{R}:=\left\{\lambda\,|\,\arg\lambda\in \left(\frac{\pi}{4},\frac{\pi}{2}\right)+\frac{m\pi}{2}, m=0,1,2,3\right\}$$ for the biharmonic Schrödinger equation (see Figure \ref{drkappa}).
\begin{figure}[h]
	\centering
	\includegraphics[width=6cm]{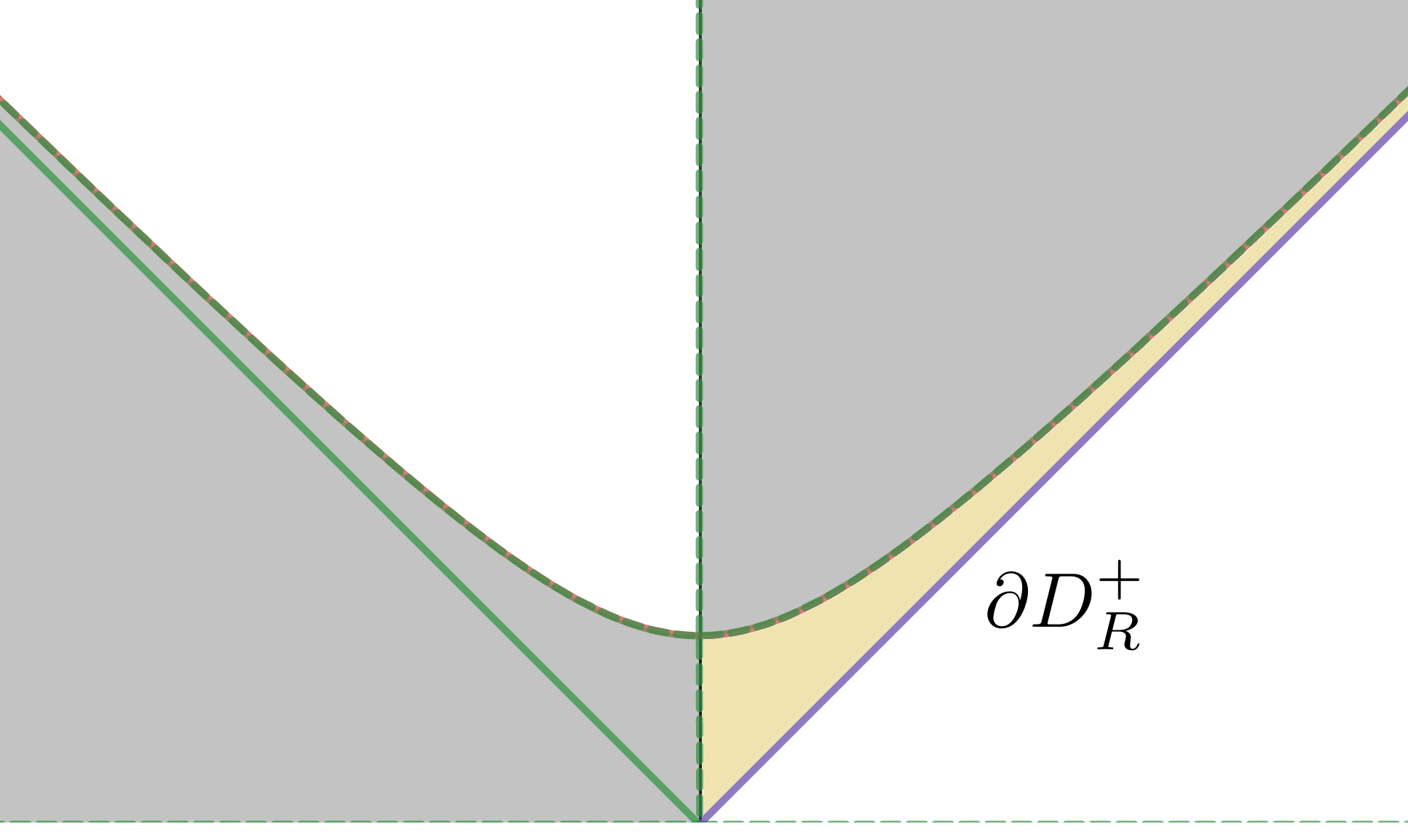}
	\caption{ }\label{drkappa}
\end{figure}
The deformation from $D_\kappa^+$ to $D_R^+$  is possible as the two regions $D_R$ and $D$ approach each other as $|\lambda|\rightarrow \infty$. Hence, the following asymptotic form of the integral representation is valid:
 \begin{equation}
\begin{split}\label{abcde}
u(x,y,t)=&\frac{1}{(2\pi)^2}\int_{-\infty}^\infty\int_{-\infty}^{\infty}E(\kappa,\lambda;x,y,t) \widehat{u_0}(\kappa,\lambda)d\lambda d\kappa\\
&+\frac{1}{(2\pi)^2}\int_{-\infty}^\infty\int_{\partial D_R^+}E(\kappa,\lambda;x,y,t)G(\kappa,\lambda,t)d\lambda d\kappa.
\end{split}
\end{equation}
\end{rem}

\begin{rem} The rigorous analysis of solutions of corresponding nonlinear IBVPs can be done in two standard stages:  (i) prove space and time estimates for the linear IBVP in function spaces (Sobolev, Bourgain, etc.) (ii) use a fixed point argument on the solution operator defined by the representation formula in which the nonhomogeneous terms are replaced by given nonlinear sources. The first stage can be carried out via a \emph{decompose-and-reunify} approach, see for instance \cite{Him17} and \cite{Ozs19}.  In this approach, one first decomposes the given linear IBVP into a homogeneous Cauchy problem, a nonhomogeneous Cauchy problem, and an inhomogeneous boundary value problem with zero initial datum and zero interior source. The integral representing the solution obtained via the UTM is very suitable for using tools from the oscillatory integral theory, Fourier and harmonic analysis to handle the inhomogeneous boundary value problem.  In fact, the explicit exponential terms in the UTM formula allows one to perform the associated wellposedness analysis in fractional function spaces and prove Strichartz type estimates. See for instance \cite[Section 3.2]{Ozs19} for a proof of a Strichartz estimate for the biharmonic Schrödinger equation that is based on an application of the Van der Corput lemma to the UTM formula.
\end{rem}
\bibliographystyle{amsplain}
\bibliography{controlrefs}
\end{document}